\documentclass{amsart}
\usepackage{stmaryrd}
\usepackage{latexsym,enumerate}
\usepackage{amssymb}
\usepackage{mathrsfs}
\usepackage[colorlinks]{hyperref}
\usepackage{todonotes}
\usepackage{placeins}

\newtheorem{theorem}{Theorem}[section]
\newtheorem{lemma}[theorem]{Lemma}
\newtheorem{corollary}[theorem]{Corollary}

\newtheorem*{thmA}{Theorem A}

\theoremstyle{definition}

\theoremstyle{remark}
\newtheorem{remark}[theorem]{Remark}

\numberwithin{equation}{section}

\usepackage{array}
\newcommand{\Aut}{{\operatorname{Aut}}}

\newcommand{\Irr}{{\operatorname{Irr}}}

\newcommand{\PSL}{{\operatorname{PSL}}}

\newcommand{\PGL}{{\operatorname{PGL}}}

\newcommand{\SL}{{\operatorname{SL}}}

\newcommand{\Out}{{\operatorname{Out}}}
\newcommand{\Alt}{{\operatorname{Alt}}}

\renewcommand{\SL}{{\mathrm {SL}}}

\DeclareMathOperator{\Sym}{Sym}

\begin{document}

\title{On groups with square-free gcd of character degree and codegree}

\author{Karam Aldahleh}
\address{Department of Mathematics, University of Rochester, 915 Hylan Building, Rochester, NY 14627, USA.}
\email{kaldahle@u.rochester.edu}

\author{Alan Kappler}
\address{Department of Mathematics, Harvey Mudd College, 301 Platt Blvd, Claremont, CA 91711, USA.}
\email{akappler@g.hmc.edu}

\author{Neil Makur}
\address{Department of Computer Science, Carnegie Mellon University, 5000 Forbes Avenue, Pittsburgh, PA 15213, USA.}
\email{nmakur@andrew.cmu.edu}

\author{Yong Yang}
\address{Department of Mathematics, Texas State University, 601 University Drive, San Marcos, TX 78666, USA.}
\makeatletter
\email{yang@txstate.edu}
\makeatother

\date{}


\begin{abstract}
Let $G$ be a finite group and $\chi$ be an irreducible character of $G$. The codegree of $\chi$ is defined as $\chi^c(1) =\frac{|G: \ker\chi|}{\chi(1)}$. In a paper by Gao, Wang, and Chen \cite{GWC24}, it was shown that $G$ cannot satisfy the condition that $\gcd(\chi(1),\chi^c(1))$ is prime for all $\chi\in\Irr(G)^\#$. We generalize this theorem by solving one of Guohua Qian's unsolved problems on character codegrees. In Qian's survey article \cite{Qiancodegree}, he inquires about the structure of non-solvable finite groups with square-free $\gcd$ instead. In particular, we prove that if $G$ is such that $\gcd(\chi(1),\chi^c(1))$ is square-free for every irreducible character $\chi$, then $G/\text{Sol}(G)$ is isomorphic to one among a particular list of almost simple groups. 
\end{abstract}

\maketitle
\Large
\section{Introduction} \label{sec:introduction8}
 Throughout this paper, $G$ is a finite group and $\Irr(G)$ is the set of irreducible characters of $G$. The degrees of these irreducible characters, defined as the dimension of their corresponding representation, have helped establish plenty of meaningful results in group theory. As an analogue to degrees, the concept of codegree was originally defined by Qian, Wang, and Wei in \cite{Qiancodegree}. For a character $\chi\in \Irr(G)$, the codegree of $\chi$ is denoted and defined as $\chi^c(1)=|G:\mathrm{ker}(\chi)|/\chi(1)$. A vast amount of research has shown that codegrees offer a new and powerful group invariant. Part of this research includes studying the divisibility relationship between the degree and codegree of a given irreducible character. In particular, this relationship was originally studied in \cite{Liang} when Liang and Qian defined and characterized $\mathscr{H}$-groups. Gao, Wang, and Chen later established in \cite{GWC24} that no group is such that $\gcd(\chi(1),\chi^c(1))$ is prime for every $\chi\in\Irr(G)^\#$. We now turn to a new problem proposed in Qian's detailed survey on character codegrees ~\cite[Problem 5.7(2)]{qsurvey}.

 We say that a group $G$ satisfies the \textbf{square-free character hypothesis} if for every $\chi\in\Irr(G)$, we have that $\gcd(\chi(1),\chi^c(1))$ is square-free. Qian's survey inquires what the structure of such a group looks like. In particular, he asks for two different descriptions in the solvable and nonsolvable case. In this paper, we resolve the latter.

 \textbf{Problem:} \textit{What is the complete description of a nonsolvable group $G$ satisfying the square-free character hypothesis?}

 Inspired by the aforementioned divisibility results, we establish the following theorem in this paper.

\begin{thmA}
    Let $G$ be a finite non-solvable group which satisfies the square-free character hypothesis. Suppose $M=\operatorname{Sol}(G)$ is the maximal normal solvable subgroup of $G$; then $G/M$ is isomorphic to one of the following:
\begin{itemize}
\item $J_1$
\item $\Alt_n$ for $n=5,6,7$
\item $\Sym_n$ for $n=5,6$
\item $M_{10}$ or $\PGL(2,9)$
\item $^2B_2(8)$ or $\Aut(^2B_2(8))$
\item A group $H$ where $\PSL(q)\le H\le\Aut(\PSL(q))$ as described in Lemma \ref{PSL-exceptions}.
\end{itemize}
\end{thmA}

\begin{remark}
    Note that $\Alt_5,\Sym_5,\Alt_6,\Sym_6,M_{10},\PGL_2(9)$ are all covered by the last case, being almost simple groups corresponding (sometimes by exceptional isomorphism) to one of the simple groups $\PSL_2(5)$ and $\PSL_2(9).$
\end{remark}

\section{Preliminaries}

\begin{lemma}[\cite{Liang}, Lemma 2.1]
    \label{codegree-divides}
    Let $G$ be a finite group and let $\chi\in\Irr(G).$
    \begin{enumerate}
        \item If $N$ be a $G$-invariant subgroup of $\ker\chi$. Then $\chi$ may be viewed as an irreducible character of $G/N$, and the codegrees of $\chi$ in $G$ and in $G/N$ are the same.
        \item If $M$ is a subnormal subgroup of G and $\psi$ is an irreducible constituent of $\chi_M$, then $\psi^c(1)\mid \chi^c(1).$
    \end{enumerate}
\end{lemma}

\begin{theorem}[Clifford's Theorem]
    \label{clifford}
    Let $\pi: G\to GL(n,K)$ be an irreducible representation over a field $K$. If $N\lhd G$, then $\pi|_N=\bigoplus_{i=1}^t\rho_i$, where each $\rho_i$ is an irreducible representation of $N$ of the same dimension. This, of course, implies the existence of some irreducible representation $\rho$ of dimension $\frac{n}{t}$.\\
    In terms of characters, if $\mu$ is a complex character of $N$, then for a fixed $g\in G$, $\mu^{(g)}$ is irreducible if and only if $\mu$ is irreducible. If $\chi\in\Irr(G)$ and $\mu\in\Irr(N)$, then when $\langle\chi|_N,\mu\rangle\neq 0$, $\chi_N=e\sum_{i=1}^t\mu^{(g_i)}$ for some $e,t\in\mathbb{Z}^+$ with $g_i\in G$.
\end{theorem}

\begin{corollary}
    Let $N\unlhd G,$ and let $\psi$ be an irreducible character of $N$; then there exists an irreducible character $\chi$ of $G$ such that $\psi(1)\mid \chi(1),$ $\psi^c(1)\mid \chi^c(1),$ and $\psi$ is an irreducible component of $\chi_N.$
\end{corollary}

\begin{theorem}[Mattarei's Lemma] Let $G$ be a finite group containing no nontrivial solvable normal subgroups. Then there are nonabelian simple groups $T_1,\ldots,T_r$ and $n_1,\ldots,n_r\geq 1$ such that $G$ contains an isomorphic copy of $T_1^{n_1}\times\cdots\times T_r^{n_r}$ and is isomorphic to a subgroup of $\Aut(T_1^{n_1}\times\cdots\times T_r^{n_r})$.
\end{theorem}
\begin{proof} Let $N_1,\ldots,N_r$ be the minimal normal subgroups of $G$, and let $N=N_1\cdots N_r$. We may write $N_i\cong T_i^{n_i}$ for some simple group $T_i$, and note that $T_i$ must be nonabelian since $N_i$ is not solvable. In particular, $Z(N_i)=1$. Now, for $i\neq j$, $N_i\cap N_j=1$, and so elements of $N_i$ and $N_j$ commute, meaning that elements of $N_i$ and $\prod_{j\neq i}N_j$ commute. Thus, $N_i\cap\prod_{j\neq i}N_j\leq Z(N_i)=1$, so that $N$ is the internal direct product of the $N_i$: $N\cong N_1\times\cdots\times N_r$.

Now, $N\unlhd G$, and so we have a homomorphism $G\to\Aut(N)$ with kernel $Z_G(N)$, where $g$ maps to $\varphi_g:n\mapsto gng^{-1}$. Since elements of $Z_G(N)$ commute with elements of $N$,
$$Z_G(N)\cap N=Z(N)\cong Z(N_1)\times\cdots\times Z(N_r)=1.$$
Note that $Z_G(N)$ is normal in $G$, as it is a kernel of a homomorphism. Thus, if $Z_G(N)\neq 1$, then it would contain a minimal normal subgroup, and so we would have $Z_G(N)\cap N\neq 1$, an impossibility. This means, $Z_G(N)=1$, and so we can embed $G$ into $\Aut(N)$.

Finally, since each $N_i$ is normal in $G$, it is fixed by the action of $G$, so that the image of $G$ lies in
$$\Aut(N_1)\times\cdots\times\Aut(N_r)\leq\Aut(T_1^{n_1}\times\cdots\times T_r^{n_r}),$$
as desired.\end{proof}

\begin{theorem}[\cite{FRT}, Theorem 1]
    If $\lambda=(\lambda_1,...,\lambda_k)$ is a finite sequence of positive integers s.t. $\lambda_1\geq\lambda_2\geq...\geq\lambda_k$, then for $n=\lambda_1+...+\lambda_k$, there exists an irreducible character $\chi_\lambda$ of $S_n$ with degree $\frac{n!}{H(\lambda)}$. Furthermore, if $\lambda\neq\overline{\lambda}$, then $\chi_\lambda$ restricted to $A_n$ is again irreducible and of the same degree. Here $H(\lambda)$ denotes the hook-length product and $\overline{\lambda}$ denotes the conjugate partition of $\lambda$.
\end{theorem}
\begin{lemma}\label{An-Chars}
    If $n\geq 8$, then $\rm{Irr}(\Alt_n)$ contains characters of degree
    $$\frac{n(n-2)(n-4)}{3},\qquad \frac{(n-1)(n-2)(n-3)}{6},\qquad \frac{n(n-3)}{2}$$ that extend to $\Sym_n$.
\end{lemma}
\begin{proof}
    Let $n\geq 8$ be arbitrary. If $\lambda=(n-3,2,1)$, notice how $\lambda$ is not self-conjugate. By Theorem 3.1, we find that $\chi_\lambda(1)=\frac{n(n-2)(n-4)}{3}$. If we alternatively let $\lambda=(n-3,1,1,1)$, then we similarly get a character corresponding to $\lambda$ with $\chi_\lambda(1)=\frac{(n-1)(n-2)(n-3)}{6}$. Lastly, let $\lambda=(n-2,2)$. It follows that $\chi_\lambda(1)=\frac{n(n-3)}{2}.$
\end{proof}

\begin{theorem}[\cite{White}, Theorem A]
    Let $S = \PSL_2(q)$, where $q=p^f>3$ for a prime $p$, $A = \Aut(S),$ and let $S\le H\le A.$ $A$ is generated by $S,$ a diagonal automorphism $\delta$ of order $2$ if $p$ is odd and $1$ otherwise, and a field automorphism $\varphi$ of order $f.$
    Set $G=\PGL_2(q)$ if $\delta\in H$ and $G = S$ if $\delta\not\in H$, and let $|H:G|=d=2^a m$, $m$ odd. If $m$ is odd, let $\varepsilon=(-1)^{(q-1)/2}$. The set of irreducible character degrees of $H$ is
    \[\operatorname{cd}(H)=\{1,q,(q+\varepsilon)/2\}\cup\{(q-1)2^a\ell:\ell\mid m\}\cup\{(q+1)j:j\mid d\},\]
    with the exceptions:
    \begin{enumerate}
        \item $(q+\varepsilon)/2$ is not a degree if $p=2$ or if $p$ is odd with $H\not\le S\langle\varphi\rangle.$
        \item $\ell\neq 1$ if $p=3,$ $f$ is odd, and $H=S\langle\varphi\rangle.$
        \item $j\neq 1$ if $p=3,$ $f$ is odd, and $H=A.$
        \item $j\neq 1$ if $p=2,3,$ or $5,$ $f$ is odd, and $H=S\langle\varphi\rangle.$
        \item $j\neq 2$ if $p=2$ or $3,$ $f$ is 2 mod 4, and $H=S\langle\varphi\rangle$ or $H=S\langle\delta\varphi\rangle.$
    \end{enumerate}
\end{theorem}

\begin{lemma}
    \label{PSL-exceptions}
    With all groups and variables as above, a subgroup $S\le H\le \Aut(S)$ has the property that $\gcd(\chi(1),\chi^c(1))$ is square-free for all $\chi\in\Irr(G)$ as long as all of the following conditions hold:
    \begin{enumerate}
        \item If $p\neq 2,$ then $p^2\nmid d$;
        \item For any odd prime $r\neq p,$ either $r^2\nmid d$ or $r^4\nmid d(q^2-1)$;
        \item $4\nmid d$ in all cases, and in particular $2\nmid d$ whenever $q\equiv 1 (\operatorname{mod}\ 8),$ $q\equiv 7 (\operatorname{mod}\ 8),$ or $q$ is odd and $\delta\in H$.
    \end{enumerate}
    These conditions are also necessary except in the two groups $\PSL_2(9)\le M_{10},S_6\le\Aut(\PSL_2(9)).$
\end{lemma}

\begin{proof}
    The value $\gcd(\chi(1),\chi^c(1))$ is square-free if and only if for every prime $p,$ $p^2$ does not divide both $\chi(1)$ and $\chi^c(1).$ Note also that since $\Out(\PSL_2(q))$ is abelian and $S$ is the unique minimal normal subgroup of $H,$ any representation is either linear or satisfies $\chi^c(1)=|G|/\chi(1).$ From here we proceed by casework.

    For an odd prime power $q=p^f,$ the character degree $q$ and its corresponding codegree have $\gcd$ a multiple of $p^2$ only when $\frac{|G|}{q},$ equal to $(q^2-1)d$ up to powers of 2, is a multiple of $p^2.$ $q,q-1,q+1$ have no pairwise common factors except for possibly 2, and in particular this condition requires $p^2\mid d.$ All other characters similarly have $p^2$ never dividing the $\gcd$ or dividing it only when $p^2\mid d.$ Conversely, the degree $q$ always works when $p^2\mid d$; for $\PSL_2(q)$ only has an extension of index $p^2$ when $q$ has more than one factor of $p,$ implying $p^2\mid\gcd(q,|G|/q).$

    For an odd prime $r\neq p,$ $r^2$ dividing the degree and codegree implies that $r^4\mid |G|,$ or (since $q$ and any factors of $2$ are relatively prime to $r^4$) that $r^4\mid (q^2-1)d.$ Furthermore, if $r^2\nmid d,$ then the remaining of $r$ must lie in $q-1$ and $q+1.$ Since these two cannot share a factor of $r,$ all but at most one factor of $r$ lies in either the degree or the codegree: for in every possible case, the character degree has all the factors of $r$ in $q-1$ but none in $q+1,$ all the factors in $q+1$ but none in $q-1,$ or neither. In all these cases, the degree or codegree without the factors of $r$ associated to $q^2-1$ will have at most one factor of $r$ in it, so that $r^2$ does not divide one of the two sides.

    Conversely, when $r^2\mid d$ and $r^4\mid (q^2-1)d,$ consider the character degrees $(q-1)2^ar^2$ and $(q+1)r^2$; these are valid since $\ell,j=1,j=2$ cannot equal $q^2.$ Since $q-1$ and $q+1$ share no factors of $r,$ one of $q-1$ and $q+1$ is not a multiple of $r.$ It follows that $r^2$ exactly divides one of the given degrees, and so the corresponding codegree is also a multiple of $r^4/r^2=r^2$. Thus their $\gcd$ is a multiple of $r^2.$

    For $q$ a power of 2, the trivial $\delta$ is in $H$ in all cases, and a degree of $q$ yields codegree $(q^2-1)d$; similarly to above, the gcd here can only be even when $4\mid d,$ as with all other degrees, and since $q>3,$ $\gcd(q,|G|/q)$ is a multiple of $4.$

    For $q$ an odd prime power, significant casework is required covering the value of $p$ mod $8$ and whether $\delta\in H.$ When $\delta\in H,$ $|G|=q(q^2-1)d,$ while when $\delta\not\in H, |G|=q(q^2-1)d/2.$ The degrees of $1,q,(q+\varepsilon)/2$ are always odd and therefore irrelevant. A case-by-case study of the possibilities for 2 gives the $2\nmid d$ and $4\nmid d$ restrictions listed above. Most of the exceptions have no effect on our square-free case, generally coinciding with already-disallowed values for $d$. Only the last one has an effect: when $p=3,f=2,j=2,d=2,$ the groups $H=S\langle\varphi\rangle$ and $H=S\langle\delta\varphi\rangle$ have no $j=2$ character, corresponding to the extensions $M_{10}$ and $S_6$ of $\PSL_2(9).$
    \begin{enumerate}
        \item When $q \equiv 1\,(\textrm{mod}\ 8)$, the codegrees corresponding to $(q-1)2^a\ell,$ namely $q(q+1)\cdot m/\ell$ or $q(q+1)/2\cdot m/\ell,$ cannot be multiples of 4. $(q+1)j$ is a multiple of $4$ if and only if $j$ is even, and the corresponding codegrees $q(q-1)\cdot d/j$ and $q(q-1)/2\cdot d/j$ are multiples of 4. Thus we have a $\gcd$ which is a multiple of 4 iff there exists an even choice for $j$. For $j$ to be even requires that $d$ be even. Conversely, if $d$ is even, we may attempt to choose $j=d$; this is only disallowed when $d=2,$ $p=2$ or $3,$ and $f$ is $2$ mod 4. For an odd prime power, this can only be $S=\PSL_2(3^{4n+2}).$ $H=S\langle\varphi\rangle$ and $H=S\langle\delta\varphi\rangle$ do not contain $\delta,$ and so the index of $S$ within them is precisely $d=4n+2.$ That is, we have exceptions to $\PSL_2(9)$ with the extensions of order $2$ not containing $\delta$; these are $M_{10}$ and $S_6$, which indeed satisfy our square-free condition.
        \item When $q \equiv 3\,(\textrm{mod}\ 8)$, $(q-1)2^a\ell$ is a multiple of 4 iff $d$ is even; when $\delta\in H,$ the codegree $q(q+1)\cdot m/\ell$ is a multiple of 4, while for $\delta\not\in H,$ $q(q+1)/2\cdot m/\ell$ is not a multiple of 4. $(q+1)j$ is always a multiple of 4 while $q(q-1)\cdot d/j$ can be a multiple of $4$ iff $d$ is even and $q(q-1)/2\cdot d/j$ can be a multiple of $4$ iff $d$ is a multiple of $4.$ Not counting exceptions, then, we have the condition $2\nmid d$ for $\delta\in H$ and $4\nmid d$ for $\delta\not\in H.$ Exceptions 2 through 4 above must be considered; since $\delta\not\in S\langle\varphi\rangle,$ we need only consider exception 3 for $\delta\in H,$ which has no effect on the $(q-1)2^a\ell$ degree that matters for $\delta\in H.$ For $\delta\not\in H,$ $\ell=1$ and $j=1$ are disallowed; we only care about $(q+1)j$ here, and when $4\mid d,$ $f$ must be even. Thus no exceptions apply here.
        \item When $q \equiv 5\,(\textrm{mod}\ 8)$, the codegrees corresponding to $(q-1)2^a\ell,$ namely $q(q+1)\cdot m/\ell$ or $q(q+1)/2\cdot m/\ell,$ cannot be multiples of 4. $(q+1)j$ is a multiple of 4 iff $j$ is even; when $\delta\in H,$ the codegree $q(q-1)\cdot d/j$ is always a multiple of 4, while for $\delta\not\in H,$ the codegree $q(q-1)/2\cdot d/j$ is a multiple of 4 iff $d/j$ is even. These correspond to the conditions $2\nmid d$ for $\delta\in H$ and $4=(2\cdot2)\nmid (j\cdot d/j)=d.$ The only exception that could apply here is exception 4, in which $p=5$ and $f$ is odd; since we care only about $j$ even, the $j\neq 1$ restriction does not matter, so no restrictions apply.
        \item When $q \equiv 7\,(\textrm{mod}\ 8)$, $(q-1)2^a\ell$ is a multiple of 4 iff $d$ is even, while the possible codegrees $q(q+1)\cdot m/\ell$ and $q(q+1)/2\cdot m/\ell$ are always multiples of 4. $(q+1)j$ is always a multiple of 4 while $q(q-1)\cdot d/j$ can be a multiple of $4$ iff $d$ is even and $q(q-1)/2\cdot d/j$ can be a multiple of $4$ iff $d$ is a multiple of $4.$ The former conditions are stronger, so the condition we want here is $2\nmid d.$
    \end{enumerate}
    These are all the necessary conditions for our $\gcd$ to be square-free.
\end{proof}

\FloatBarrier

\begin{table}
\setlength\extrarowheight{4pt}
\begin{tabular}{c c|c c|c c}
$H$ & $\chi(1)$ & $H$ & $\chi(1)$ & $H$ & $\chi(1)$ \\
\hline
$M_{11}$ & $2^2\cdot11$ & $He$ & $2^3\cdot5\cdot17$ & $Fi_{23}$ & $2^2\cdot3\cdot13\cdot23$ \\
$M_{12}$ & $2^4$ & $Ru$ & $2^2\cdot3^2\cdot7\cdot13$ & $Co_1$ & $2^2\cdot3\cdot23$ \\
$M_{22}$ & $2^3\cdot5\cdot7$ & $Suz$ & $2^2\cdot7\cdot13$ & $J_4$ & $2^3\cdot3^2\cdot23\cdot29\cdot37$ \\
$J_2$ & $2^2\cdot 3^2$ & $ON$ & $2^6\cdot3^2\cdot19$ & $Fi_{22}$ & $2^5\cdot3\cdot5\cdot7\cdot13$ \\
$M_{23}$ & $2^3\cdot11\cdot23$ & $Co_3$ & $2^7\cdot 7$ & $Fi_{24}$ & $2^3\cdot7^2\cdot11\cdot17\cdot23\cdot29$ \\
$HS$ & $2^7\cdot 7$ & $Co_2$ & $2^3\cdot11\cdot23$ & $B$ & $3^2\cdot5\cdot23\cdot31$ \\
$J_3$ & $2^2\cdot 3^4$ & $HN$ & $2^3\cdot5\cdot19$ & $M$ & $2^2\cdot31\cdot41\cdot59\cdot71$ \\
$M_{24}$ & $2^2\cdot 3^2\cdot 7$ & $Ly$ & $2^4\cdot5\cdot31$ \\
$McL$ & $2^2\cdot3^2\cdot7$ & $Th$ & $2^3\cdot31$ \\
\end{tabular}
\caption{Character degrees for sporadic groups such that $\gcd(\chi(1),\chi^c(1))$ is not square-free}
\label{Sporadic-Groups-Table}
\end{table}

\section{Proof}
\subsection{Theorem A Proof} By Mattarei's Lemma, we may write $G\leq\Aut(H)$, where there exiss (not necessarily distinct) non-abelian simple groups $S_1,...,S_t$ such that $H=\prod_{i=1}^t S_i$. We claim that $t=1.$ Suppose $t>1,$ and consider in particular the characters of $S_1$ and $S_2.$ Note that every non-abelian simple group has a nontrivial irreducible character $\alpha_{\text{odd}}$ with odd degree (see Corollary 12.2 of \cite{Isaacs/book}). Furthermore, every non-abelian simple group except $A_7$ has a character $\chi$ with degree a multiple of 4, by \cite{Lewis}. If at least one of $S_1$ and $S_2$ is not $A_7,$ say $S_1$ without loss of generality, then consider the character $\psi=\chi\times\alpha_\text{odd}\times1\times\cdots,$ where all remaining characters are trivial. Then both the $\psi(1)$ and $\psi^c(1)$ are multiples of 4 (recall that nonabelian simple groups have order divisible by $4$). When $S_1=S_2=A_7,$ let $\tau$ be a character with even degree and codegree, and similarly use the character $\tau\times\tau\times1\times\cdots.$ In both cases, noting that both character degrees and codegrees are multiplicative over groups with trivial center, our $\gcd$ is not square-free. Hence, $t=1$.

Then $H$ is a non-abelian simple group and $G/M$ is almost simple. Since $H\unlhd G/M,$ for any $\chi\in\Irr(G/M)$ and $\psi$ an irreducible component of $\chi_N,$ by Clifford's theorem $\psi(1)\mid \chi(1),$ and by \ref{codegree-divides} $\psi^c(1)\mid\chi^c(1).$ Furthermore, we can construct for any such $\psi$ some $\chi\in\Irr(G)$ such that $\langle\chi_N,\psi\rangle\neq 0,$ from which it follows that $\chi_N$ has $\psi$ as a constituent.
Therefore, if $H$ does not satisfy the square-free condition, $G/M$ does not either. Our approach will then to be to find all $H$ satisfying the square-free condition, then to check all corresponding almost simple groups. 

We first check all possible choices for simple groups $H$:
\begin{itemize}
\item Suppose that $H$ is sporadic and $\neq J_1$. Then Table \ref{Sporadic-Groups-Table} shows that $H$ does not satisfy the condition. Since $|J_1|=2^3\cdot3\cdot5\cdot7\cdot11\cdot19$ is $4$th power-free, $J_1$ must satisfy the square-free condition.
\item Suppose that $H\cong \Alt_n$ for $n\geq 5$. If $n=5,6,7$, then $|\Alt_n|$ is $4$th power-free, and so $\Alt_n$ satisfies the square-free hypothesis. If $n=8$, then Lemma \ref{An-Chars} shows that $\Alt_n$ contains a character of degree $20$, which has codegree $1008$, and thus $\Alt_8$ does not satisfy the square-free hypothesis. If $n>8$ is odd, then Lemma \ref{An-Chars} gives a character of degree $\frac{(n-1)(n-2)(n-3)}{6}$ (which is divisible by $4$), which has codegree $3n\cdot(n-4)!$, which is divisible by $4$. If $n>8$ is even, then Lemma \ref{An-Chars} gives a character of degree $\frac{n(n-2)(n-4)}{3}$ (which is divisible by $4$), which has codegree $\frac{3}{2}(n-1)(n-3)(n-5)!$, which is divisible by $4$. Thus, $\Alt_n$ does not satisfy the square-free condition for $n\geq 8$.
\item Suppose that $H$ is simple of Lie type and not $^2B_2(8)$ or of the form $A_1(q)$. Then Tables \ref{Lie-Groups-Table} and \ref{Lie-Groups-Table2} show that $H$ does not satisfy the condition. $A_1(q),$ on the other hand, has character degrees $1,q-1,q,q+1,$ and possibly $\frac{q+(-1)^{(q-1)/2}}{2}.$ All of these yield a $\gcd$ of 1 or 2, meaning that $A_1(q)$ satisfies our condition. $^2B_2(8)$ also satisfies the condition by manual checking of characters.
\end{itemize}

It now remains only to check the corresponding almost simple groups. The almost simple groups corresponding to $A_1(q)\cong\PSL_2(q)$ are covered in Lemma \ref{PSL-exceptions}. $J_1$ has trivial outer automorphism group and is therefore the only almost simple group of its type. $^2B_2(8)$ has outer automorphism group of order 3, and so the only two almost simple groups corresponding to it are itself and $\Aut(^2B_2(8))$; both can be confirmed to satisfy the square-free condition. $\Alt_5$ and $\Alt_7$ have themselves and $\Sym_5$ and $\Sym_7$; $\Sym_5$ satisfies the square-free condition while $\Sym_7$ does not. Finally, the groups lying between $\Alt_6$ and its automorphism group are $\Alt_6,\Sym_6,M_{10},\PGL_2(9),\Aut(\Alt_6)$; the first four do satisfy the square-free condition while $\Aut(\Alt_6)$ does not. These are all the possible values for $G/M.$ \qed

\FloatBarrier

\begin{table}
\setlength\extrarowheight{8pt}
\begin{tabular}{c|c|c}
$H$ & $\chi(1)$ & Factor of $\gcd$\\
\hline
$A_2(3)$ & $2^2\cdot3$ & 4 \\
$A_2(4)$ & $2^2\cdot5$ & 4 \\
$A_2(q),\;q>4,\;q\equiv0$ mod 4 & $q(q^2+q+1)$ & 4 \\
$A_2(q),\;q>4,\;q\equiv1$ mod 4 & $(q-1)(q^2+q+1)$ & 4 \\
$A_2(q),\;q>4,\;q\equiv3$ mod 4 & $(q+1)(q^2+q+1)$ & 4 \\
$A_3(q)$ & $q^3(q^2+q+1)$ & $q^2$ \\
$A_4(q)$ & $q^6(q+1)(q^2+1)$ & $q^2$ \\
$A_n(q),\; n>4$ & $\dfrac{q^3(q^{n-1}-1)(q^n-1)}{(q-1)(q^2-1)}$ & $q^2$\\
$B_2(q),\; q$ odd & $(q-1)(q^2+1)$ & 4\\
$B_2(q),\; q$ even, $q>2$ & $q(q+1)(q^2+1)$ & 4\\
$B_n(q),\; n>2$ & $\dfrac{(q^{n-2}-1)(q^{n-1}-1)(q^{n-1}+1)(q^n+1)}{2(q^2-1)^2}$ & $q^2$\\
$C_n(q),\; n>2$ & $\dfrac{q^4(q^{2(n-2)}-1)(q^{2n}-1)}{2(q^2-1)^2}$ & $q^2$\\
$D_4(2)$ & $2^2\cdot7$ & $4$\\
$D_4(3)$ & $2^2\cdot5\cdot13$ & $4$\\
$D_4(q),\;q>3$ & $\dfrac{1}{2}q^3(q+1)^4(q^2-q+1)$ & $q^2$\\
$D_5(2)$ & $2^2\cdot5\cdot17$ & $4$\\
$D_5(q),\;q>2$ & $q^2(q^2+1)(q^4+1)$ & $q^2$\\
$D_n(q),\; n>5$ & $\dfrac{q^6(q^{n-4}+1)(q^{2(n-3)}-1)(q^{2(n-1)}-1)(q^n-1)}{(q^2-1)^2(q^4-1)}$ & $q^2$ \\
\end{tabular}
\vspace{10pt}
\caption{Character degrees for groups of Lie type such that $\gcd(\chi(1),\chi^c(1))$ is not square-free.}
\label{Lie-Groups-Table}
\end{table}

\begin{table}
\begin{tabular}{c|c|c}
$H$ & $\chi(1)$ & Factor of $\gcd$\\
\hline
$E_6(q)$ & $q^4\phi_2^3\phi_4^2\phi_6^2\phi_8\phi_{12}$ & $q^2$ \\
$E_7(q)$ & $\frac{1}{2}q^3\phi_1^4\phi_3^2\phi_5\phi_7\phi_9\phi_{14}$ & $q^2$  \\
$E_8(q)$ & $\frac{1}{2}q^4\phi_1^4\phi_3^2\phi_4^2\phi_5^2\phi_7\phi_9\phi_{10}\phi_{12}\phi_{15}\phi_{20}\phi_{30}$& $q^2$ \\
$F_4(q)$ & $q^2\phi_3^2\phi_6^2\phi_{12}$ & $q^2$\\
$G_2(3)$ & $2^6\cdot3^6\cdot7\cdot13$ & $4$ \\
$G_2(4)$ & $2^{12}\cdot3^3\cdot5^2\cdot7\cdot13$ & $4$ \\
$G_2(q),\;q>4,\;q\equiv 1$ mod 2 & $q^2(q^2-q+1)(q^2+q+1)$ & $q^2$\\
$G_2(q),\;q>4,\;q\equiv 2$ mod 6 & $q^3(q-1)(q^2+q+1)$ & $q^2$\\
$G_2(q),\;q>4,\;q\equiv 4$ mod 6 & $q^3(q+1)(q^2-q+1)$ & $q^2$\\
$^2A_2(q),\;q>2,\;q\equiv0$ mod 4 & $q(q^2-q+1)$ & 4 \\
$^2A_2(q),\;q>2,\;q\equiv1$ mod 4 & $(q-1)(q^2-q+1)$ & 4 \\
$^2A_2(q),\;q>2,\;q\equiv3$ mod 4 & $(q+1)(q^2-q+1)$ & 4 \\
$^2A_3(q)$ & $q^2(q^2+1)$ & $q^2$ \\
$^2A_n(q),\; n>3$ & $\dfrac{q^3(q^{n-1}-(-1)^{n-1})(q^n-(-1)^n)}{(q+1)(q^2-1)}$ & $q^2$ \\
$^2B_2(q),\;q=2^{2n+1},\;n>1$ & $(q-1)(q/2)^{1/2}$ & $4$\\
$^2D_4(q)$ & $\frac{1}{2}q^3(q+1)^4(q^2-q+1)$ & $q^2$\\
$^2D_n(q),\; n>5$ & $\dfrac{q^6(q^{n-4}+1)(q^{2(n-3)}-1)(q^{2(n-1)}-1)(q^n-1)}{(q^2-1)^2(q^4-1)}$ & $q^2$ \\
$^3D_4(q^3)$ & $q^7\phi_{12}$ & $q^2$\\
$^2E_6(q^2)$ & $q^6\phi_3^2\phi_6^3\phi_{12}\phi_{18}$ & $q^2$ \\
$^2F_4(2)'$ & $2^2\cdot3\cdot5^2$ & $4$\\
$^2F_4(q),\;q=2^{2n+1},\;n>0$ & $q\phi_{6}\phi_{12}$ & $4$\\
$^2G_2(27)$ & $3^3\cdot19\cdot37$ & $9$\\
$^2G_2(q),\;q=3^{2n+1},\;n>1$ & $(q/3)^{1/2}\phi_1\phi_2\phi_4$ & $9$\\
\end{tabular}
\vspace{2pt}
\caption{Character degrees for groups of Lie type continued.}
\label{Lie-Groups-Table2}
\end{table}

\FloatBarrier

\section*{Acknowledgements}
This research was conducted by Karam Aldahleh, Alan Kappler, and Neil Makur during the Summer of 2025 under the supervision of Yong Yang. The work was supported by the NSF under grants DMS-2150205 and DMS-2447229.

The authors gratefully acknowledge the financial support of NSF and thank Texas State University for providing a great working environment and support. Yang was also partially supported a grant from the Simons Foundation (\#918096).

\section*{Disclosure Statement}
The authors declare that there is no potential conflict of interest that could influence this paper. The authors declare that there is no potential competing interest that could influence this paper.

\section*{Data Availability Statement} Data sharing is not applicable to this article, as no data sets were generated or analysed during the current study.


\end{document}